\documentclass[11pt]{article}
\usepackage{amsmath,amsfonts,amsthm,amssymb,mathtools}
\usepackage{mathrsfs,graphicx,color,latexsym,tikz,calc}
\usepackage{tikz-cd}
\usepackage{epstopdf}
\usepackage{enumerate}
\usepackage{caption}
\IfFileExists{stmaryrd.sty}{\usepackage{stmaryrd}}{}
\usepackage{xcolor}
\usepackage{hyperref}
\hypersetup{
    colorlinks=true,
    citecolor=blue,
    linkcolor=blue,
    urlcolor=blue
}

\usetikzlibrary{shadows}
\usetikzlibrary{patterns,arrows,decorations.pathreplacing}

\usepackage[margin=2.5cm]{geometry}
\IfFileExists{ulem.sty}{\usepackage{ulem}}{}
\allowdisplaybreaks[4]

\newtheorem{theorem}{Theorem}[section]
\newtheorem{lemma}[theorem]{Lemma}

\newtheorem{conjecture}[theorem]{Conjecture}

\title{\Large\bfseries The Keevash--Mubayi simplex-cluster conjecture%
\thanks{Lihua Feng was supported by the NSFC (Nos. 12271527 and 12471022)
and NSF of Qinghai Province (No. 2025-ZJ-902T).
E-mail addresses: \url{wuyjmath@163.com} (Y. Wu),
\url{fenglh@163.com} (L. Feng).}}
\author{
{\small Yongjiang Wu,\quad Lihua Feng\thanks{Corresponding author.}}\\[2mm]
\small School of Mathematics and Statistics, HNP-LAMA,
Central South University\\
\small Changsha, Hunan, 410083, China
}
\date{\today}

\begin{document}
\maketitle

\begin{abstract}
A $d$-simplex-cluster is a collection of $d+1$ distinct
$k$-element sets with empty total intersection, nonempty
intersection for every $d$ members, and union of size at most $2k$.
We prove the simplex-cluster conjecture of Keevash and Mubayi,
a common strengthening of the Erd\H{o}s--Chv\'atal simplex
conjecture and Mubayi's cluster conjecture.
More precisely, for integers $k>d\ge2$ and $n\ge k(d+1)/d$,
every family of $k$-element subsets of an $n$-element set
containing no $d$-simplex-cluster has at most
$\binom{n-1}{k-1}$ members.
Equality holds if and only if the family consists of all
$k$-element subsets containing a fixed point.
\end{abstract}

\noindent\hspace*{1.5em}\textbf{AMS Classification:} 05D05; 05C65\\
\hspace*{1.5em}\textbf{Keywords:} Simplex-clusters;
Keevash--Mubayi conjecture; Erd\H{o}s--Ko--Rado theorem;
Spectral methods.

\section{Introduction}\label{sec:introduction}

For integers $n\ge1$ and $k\ge0$, write $[n]=\{1,\ldots,n\}$
and let $\binom{[n]}{k}$ denote the family of all $k$-element
subsets of $[n]$.
A family $\mathcal F$ of subsets of $[n]$ is
\textit{$k$-uniform} if every member has $k$ elements,
and \textit{intersecting} if $A\cap B\ne\emptyset$
for all $A,B\in\mathcal F$.
For $k\ge1$, the family of all $k$-element subsets of $[n]$
containing a fixed point is called a \textit{full star}.
It is intersecting and has $\binom{n-1}{k-1}$ members.
The Erd\H{o}s--Ko--Rado theorem~\cite{EKR} states that,
for $n\ge2k\geq 2$, every intersecting family
$\mathcal F\subseteq\binom{[n]}{k}$ satisfies
\[
 |\mathcal F|\le\binom{n-1}{k-1}.
\]
If $n>2k$, equality holds if and only if $\mathcal F$ is a full star.
Erd\H{o}s~\cite{Erdos1971} asked whether the same bound
holds for families containing no \textit{triangle},
a collection of three distinct sets with empty total
intersection but nonempty pairwise intersections.
Chv\'atal~\cite{Chvatal} generalized this question to
\textit{$d$-simplices}, collections $\{A_1,\ldots,A_{d+1}\}$
of $d+1$ distinct sets satisfying
\[
 \bigcap_{i=1}^{d+1}A_i=\emptyset
 \qquad\text{and}\qquad
 \bigcap_{i\in[d+1],i\ne j}A_i\ne\emptyset
 \quad\text{for every }j\in[d+1].
\]
The Erd\H{o}s--Chv\'atal simplex
conjecture~\cite{Chvatal,KM} asserts that if
$k>d\ge1$ and $n>k(d+1)/d$, then every family
$\mathcal F\subseteq\binom{[n]}{k}$ containing no
$d$-simplex satisfies
\[
 |\mathcal F|\le\binom{n-1}{k-1},
\]
with equality if and only if $\mathcal F$ is a full star.
The case $d=1$ follows from the Erd\H{o}s--Ko--Rado
theorem, since a $1$-simplex is a pair of nonempty
disjoint sets.

A related configuration is a \textit{$d$-cluster},
a collection
$\{A_1,\ldots,A_{d+1}\}\subseteq\binom{[n]}{k}$
of $d+1$ distinct sets satisfying
\[
 \bigcap_{i=1}^{d+1}A_i=\emptyset
 \qquad\text{and}\qquad
 \left|\bigcup_{i=1}^{d+1}A_i\right|\le2k.
\]
For $d=1$, the forbidden collection is a pair of disjoint
$k$-element sets, so the Erd\H{o}s--Ko--Rado theorem
determines the maximum size of a family containing
no $1$-cluster.
Mubayi~\cite{Mubayi,KM} proved the corresponding result
for $d=2$ and conjectured that, for $k>d\ge1$
and $n>k(d+1)/d$, every family
$\mathcal F\subseteq\binom{[n]}{k}$ containing no
$d$-cluster satisfies
\[
 |\mathcal F|\le\binom{n-1}{k-1},
\]
with equality if and only if $\mathcal F$ is a full star.

A collection that is both a $d$-simplex and a
$d$-cluster is called a \textit{$d$-simplex-cluster}.
For $d=1$, the three configurations coincide, each
consisting of two disjoint $k$-element sets.
In general, every family containing no $d$-simplex
or no $d$-cluster also contains no $d$-simplex-cluster.
The converse need not hold, since a family containing
no $d$-simplex-cluster may contain both $d$-simplices
with union larger than $2k$ and $d$-clusters that are
not $d$-simplices.
Keevash and Mubayi~\cite{KM} proposed the following
common strengthening of the simplex and cluster
conjectures.

\begin{conjecture}[Keevash--Mubayi~\cite{KM}]
Let $k>d\ge2$ and $n>k(d+1)/d$.
If $\mathcal F\subseteq\binom{[n]}{k}$ contains no
$d$-simplex-cluster, then
\[
 |\mathcal F|\le\binom{n-1}{k-1},
\]
with equality only for a full star.
\end{conjecture}

The two original conjectures have now been proved.
For the simplex conjecture, Chv\'atal~\cite{Chvatal}
proved the case $k=d+1\ge3$ for $n\ge k+2$.
Frankl~\cite{FranklSimplex} obtained an asymptotically
sharp bound as $n\to\infty$ for every fixed
$k>d\ge2$.
Frankl and F\"uredi~\cite{FFSimplex} then proved
the conjecture for every fixed $k>d\ge2$ and
sufficiently large $n$.
Mubayi and Verstra\"ete~\cite{MV} proved the triangle
case $d=2$ for all $k>2$ and $n\ge3k/2$.
Mubayi~\cite{MubayiStability} subsequently obtained
a stability theorem for each fixed $k>2$ and
sufficiently large $n$.
Keevash and Mubayi~\cite{KM} used stability methods
to prove both the simplex and cluster conjectures for each fixed $d\ge2$
and $\zeta>0$ when
$\zeta n<k<n/2-T(d,\zeta)$ and $n$ is sufficiently
large.
Keller and Lifshitz~\cite{KellerLifshitz} proved
the simplex conjecture for each fixed $d\ge2$
and all $k$ satisfying $d<k\le dn/(d+1)$,
provided that $n>n_0(d)$.
In a recent preprint, Huang~\cite{Huang} proved
the conjecture for all $k>d\ge2$ and
$n\ge k(d+1)/d$.

For the cluster conjecture, Frankl and
F\"uredi~\cite{FFCluster} proved the case $d=2$
for $k>2$ and $n\ge k^2+3k$.
Mubayi~\cite{Mubayi} extended this result to all
$k>2$ and $n\ge3k/2$.
He later proved the case $d=3$ for each fixed
$k>3$ and sufficiently large
$n$~\cite{MubayiFour}.
Mubayi and Ramadurai~\cite{MubayiRamadurai} and,
independently, F\"uredi and \"Ozkahya~\cite{FurediOzkahya}
proved the conjecture for every fixed $k>d\ge2$
and sufficiently large $n$.
Currier~\cite{CurrierCluster} subsequently proved
it for all  admissible parameters.

For the stronger simplex-cluster conjecture,
Lifshitz~\cite{Lifshitz} proved the conjecture
for each fixed $d\ge2$ and $\zeta>0$ in the range
$\zeta n<k\le dn/(d+1)$, provided that
$n\ge n_0(d,\zeta)$.
Currier~\cite{Currier} proved it for
$k>d\ge3$ and $n\ge2k-d+2$.
In this paper, we prove the simplex-cluster
conjecture in its full parameter range.

\begin{theorem}\label{thm:main}
Let $k>d\geq 2$ and $n\ge k(d+1)/d$.
If $\mathcal F\subseteq\binom{[n]}k$ contains no $d$-simplex-cluster,
then
\[
 |\mathcal F|\le\binom{n-1}{k-1}.
\]
Equality holds if and only if, for some $x\in[n]$,
$\mathcal F=\{A\in\binom{[n]}k:x\in A\}$.
\end{theorem}

We prove Theorem~\ref{thm:main} by induction on $d$. For $d=2$, we refine a cross-intersection inequality used by Currier~\cite{CurrierCluster} by adapting the cyclic interval argument of Borg and Leader~\cite{BL}. For $d\ge3$ and $n\le2k-d$, we pass to the complementary family and compare two bounds on the number of disjoint pairs. The induction hypothesis gives an upper bound, while the least eigenvalue of the Kneser graph gives a lower bound in terms of the family size. Together, these bounds yield the desired estimate. Currier's theorem~\cite{Currier} covers the range $n\ge2k-d+2$. For $n=2k-d+1$, we use local counting, with two exceptional cases treated directly.

The remainder of the paper is organized as follows. In Section~\ref{sec:marked}, we establish the refined cross-intersection inequality. In Section~\ref{sec:triangle}, we apply this inequality to prove the 2-simplex-cluster bound and characterize the families attaining equality. Finally, in Section~\ref{sec:completion}, we complete the proof of Theorem~\ref{thm:main}.

\section{A refined cross-intersection inequality}
\label{sec:marked}

Two families $\mathcal A,\mathcal B\subseteq\binom{[n]}{k}$
are \textit{cross-intersecting} if $A\cap B\ne\emptyset$
for every $A\in\mathcal A$ and $B\in\mathcal B$.
We use the following cross-intersection inequality
in the form given by Currier~\cite[Theorem~3]{CurrierCluster};
see also~\cite{Borg,BL}.

\begin{lemma}\label{lem:markedbasic}
Let $k\ge1$ and $n\ge2k$.
Suppose that
$\mathcal G\subseteq\mathcal H\subseteq\binom{[n]}{k}$
and that $\mathcal G$ and $\mathcal H$
are cross-intersecting.
Put $\mathcal H^*=\mathcal H\setminus\mathcal G$.
Then
\begin{equation}\label{eq:markedbasic}
 |\mathcal H|
 \le \binom{n-1}{k-1}
   +\frac{n-k}{n}|\mathcal H^*|.
\end{equation}
\end{lemma}

We next refine this inequality under an additional condition,
using the cyclic interval argument of Borg and Leader~\cite{BL}.

\begin{lemma}\label{lem:markednew}
Let $k\ge2$ and $n\ge2k+2$.
Suppose that
$\mathcal G\subseteq\mathcal H\subseteq\binom{[n]}{k}$
and that $\mathcal G$ and $\mathcal H$
are cross-intersecting.
Put $\mathcal H^*=\mathcal H\setminus\mathcal G$,
and let $\mathcal T\subseteq\mathcal H^*$.
Suppose that for each $R\in\mathcal T$
there is a point $y\in[n]\setminus R$ such that,
for every $a\in R$, the only members of $\mathcal H$
containing $R\setminus\{a\}$ are
\[
 R
 \quad\text{and}\quad
 (R\setminus\{a\})\cup\{y\}.
\]
Writing $t=|\mathcal T|$, we have
\begin{equation}\label{eq:linkform}
 |\mathcal H|
 \le \binom{n-1}{k-1}
   +\frac{n-k}{n}|\mathcal H^*|
   -\frac{t}{n}.
\end{equation}
The inequality is strict if $n\ge2k+3$ and $t>0$.
\end{lemma}

\begin{proof}
Fix a cyclic ordering $(x_0,\ldots,x_{n-1})$ of $[n]$,
with subscripts taken modulo $n$.
Its $k$-element intervals are defined by
\[
 I_i=\{x_i,x_{i+1},\ldots,x_{i+k-1}\},
 \qquad 0\le i\le n-1.
\]
Let $j$, $s$, and $r$ be the numbers of these intervals
belonging to $\mathcal G$, $\mathcal H^*$,
and $\mathcal T$, respectively.
Put
\[
 L=n-\frac nk j-s.
\]
We show that $L\ge r/k$ for every cyclic ordering.

Suppose first that $j\ge1$, and choose
$
 I_a=\{x_a,x_{a+1},\ldots,x_{a+k-1}\}\in\mathcal G.
$
Set $\widetilde x_i=x_{a+i}$ for $0\le i\le n-1$,
with subscripts taken modulo $n$.
Then $(\widetilde x_0,\ldots,\widetilde x_{n-1})$
describes the same cyclic ordering, starting at $x_a$,
and the chosen interval is
$\{\widetilde x_0,\ldots,\widetilde x_{k-1}\}$.
The counts $j,s,r$ are unchanged.
We now omit the tildes, so that
\[
 I_0=\{x_0,\ldots,x_{k-1}\}\in\mathcal G.
\]
Every interval from $\mathcal H$ must intersect $I_0$.
An interval starting at an index between $k$ and $n-k$
is disjoint from $I_0$.
Hence, the possible starting indices of intervals from
$\mathcal H$ are
\[
 -k+1,\ldots,-1,0,1,\ldots,k-1 \pmod{n}.
\]
Apart from $I_0$, these possible intervals form
the following $k-1$ pairs:
\[
 (I_i,I_{i+k}),
 \qquad -k+1\le i\le-1.
\]
The two intervals in each pair are disjoint.
If one interval in a pair belongs to $\mathcal G$,
its partner cannot belong to $\mathcal H$.
Since $\mathcal G\subseteq\mathcal H$, every interval
from $\mathcal G$ other than $I_0$ occurs in one of the
$k-1$ pairs.
If a pair contains an interval from $\mathcal G$,
its partner cannot belong to $\mathcal H$.
Thus, the other $j-1$ intervals from $\mathcal G$
occupy distinct pairs, so $j-1\le k-1$ and hence $j\le k$.

These $j-1$ pairs contain no interval from $\mathcal H^*$,
and neither does $I_0$.
All intervals from $\mathcal H^*$ must occur
in the remaining
\[
 (k-1)-(j-1)=k-j
\]
pairs, with at most two in each pair.
Hence, $s\le2(k-j)$.
Since $\mathcal T\subseteq\mathcal H^*$, we also have
$r\le s\le2(k-j)$.
Using $n\ge2k+2$, we obtain
\begin{equation}\label{eq:cyclepositive}
 \begin{aligned}
 L
 &=n-\frac nk j-s\\
 &\ge n-\frac nk j-2(k-j)
   =\frac{n-2k}{k}(k-j)\\
 &\ge\frac{n-2k}{2k}r
   \ge\frac rk.
 \end{aligned}
\end{equation}

Now suppose that $j=0$.
Every interval belonging to $\mathcal H$ then belongs
to $\mathcal H^*$, so $L=n-s$ is the number of intervals
absent from $\mathcal H$.
If $r=0$, then $L\ge0=r/k$, as required.
Assume that $r>0$.
Let $R=I_i\in\mathcal T$ be any such interval.
Its two adjacent intervals are
\[
 \begin{aligned}
 I_{i-1}
 &= (R\setminus\{x_{i+k-1}\})\cup\{x_{i-1}\},\\
 I_{i+1}
 &= (R\setminus\{x_i\})\cup\{x_{i+k}\}.
 \end{aligned}
\]
The added points $x_{i-1}$ and $x_{i+k}$ are distinct
and lie outside $R$.
If both adjacent intervals belonged to $\mathcal H$,
the hypothesis on $R$ would force both added points
to equal its associated point $y$.
Thus, at least one adjacent interval is absent
from $\mathcal H$.

Count pairs consisting of an interval from $\mathcal T$
and an adjacent interval absent from $\mathcal H$.
Each of the $r$ intervals from $\mathcal T$
occurs in at least one such pair.
Conversely, an absent interval $I_a$ can be paired
only with $I_{a-1}$ or $I_{a+1}$.
Thus, each of the $L$ absent intervals occurs
in at most two such pairs.
Hence, $r\le2L$.
Since $k\ge2$, we obtain
\[
 L\ge\frac r2\ge\frac rk.
\]

We now average over all $(n-1)!$ cyclic orderings of $[n]$.
Let $\Omega$ be the set of cyclic orderings of $[n]$.
For $\sigma\in\Omega$, let $j_\sigma$, $s_\sigma$,
and $r_\sigma$ denote the corresponding values
of $j$, $s$, and $r$.
Since each $k$-element set is an interval
in exactly $k!(n-k)!$ cyclic orderings,
double counting gives
\[
 \sum_{\sigma\in\Omega}j_\sigma
 =|\mathcal G|k!(n-k)!,\quad
 \sum_{\sigma\in\Omega}s_\sigma
 =|\mathcal H^*|k!(n-k)!,\quad
 \sum_{\sigma\in\Omega}r_\sigma
 =tk!(n-k)!.
\]
Since $L\ge r/k$ for every cyclic ordering, we have
\[\frac nk j_\sigma+s_\sigma+\frac1k r_\sigma\le n.\]
Hence,
\[
 \left(
 \frac nk|\mathcal G|+|\mathcal H^*|+\frac tk
 \right)k!(n-k)!
 =\sum_{\sigma\in\Omega}
   \left(\frac nk j_\sigma+s_\sigma+\frac1k r_\sigma\right)\le n!.
\]
Dividing by $k!(n-k)!$ and multiplying by $k$ yields
\[
 n|\mathcal G|+k|\mathcal H^*|+t
 \le k\binom{n}{k}.
\]
Since $|\mathcal G|=|\mathcal H|-|\mathcal H^*|$,
it follows that
\[
 |\mathcal H|
 \le \binom{n-1}{k-1}
   +\frac{n-k}{n}|\mathcal H^*|
   -\frac tn,
\]
which proves \eqref{eq:linkform}.

Finally, suppose that $n\ge2k+3$ and $t>0$. Since $t=|\mathcal T|>0$, choose $R\in\mathcal T$.
Let $y$ be the point associated with $R$ in the hypothesis.
Take distinct points $p,q\in[n]\setminus(R\cup\{y\})$.
Choose a cyclic ordering in which the elements of $R$
are consecutive, with $p$ immediately before them
and $q$ immediately after them.
Write $R=\{x_1,\ldots,x_k\}$ in its cyclic order.
The two intervals adjacent to $R$ are
\[
 (R\setminus\{x_k\})\cup\{p\}
 \quad\text{and}\quad
 (R\setminus\{x_1\})\cup\{q\}.
\]
By the hypothesis on $R$ and $p,q\ne y$,
neither belongs to $\mathcal H$.
Moreover, $R\in\mathcal T$ is an interval in this ordering,
so $r\ge1$.

If $j=0$, then $L$ is the number of intervals outside $\mathcal H$.
Count pairs $(A,B)$, where $A\in\mathcal T$
and $B\notin\mathcal H$ is adjacent to $A$.
As shown above, every interval from $\mathcal T$
has at least one neighbour outside $\mathcal H$.
The chosen interval $R$ has two such neighbours.
Hence, there are at least $(r-1)+2=r+1$ pairs.
Each interval outside $\mathcal H$ has only two neighbours,
so it occurs in at most two pairs.
Hence, $r+1\le2L$.
Since $k\ge2$, we obtain
\[
 L\ge\frac{r+1}{2}>\frac r2\ge\frac rk.
\]
If $j\ge1$, then \eqref{eq:cyclepositive} gives
\[
 L\ge\frac{3r}{2k}>\frac rk.
\]
Thus, $L\ge r/k$ holds for every cyclic ordering
and is strict for at least one.
Averaging gives strict inequality in \eqref{eq:linkform}.
\end{proof}

\section{Families without $2$-simplex-clusters}\label{sec:triangle}

Using Lemma~\ref{lem:markednew}, we prove the following
bound for families containing no $2$-simplex-cluster
when $k\ge4$ and $n>2k$.
We treat the remaining cases for $d=2$
in Section~\ref{sec:completion}.

\begin{theorem}\label{thm:triangle}
Let $k\ge4$ and $n>2k$.
If $\mathcal F\subseteq\binom{[n]}{k}$
contains no $2$-simplex-cluster, then
\[
 |\mathcal F|\le\binom{n-1}{k-1}.
\]
Equality holds if and only if $\mathcal F$ is a full star.
\end{theorem}

Before proving Theorem~\ref{thm:triangle},
we establish two lemmas concerning the degrees of $(k-1)$-sets.
We first introduce the notation used in their statements.
For $\mathcal F\subseteq\binom{[n]}{k}$ and
$T\in\binom{[n]}{k-1}$, define
\[
 \deg_{\mathcal F}(T)
 =|\{A\in\mathcal F:T\subseteq A\}|.
\]
For each $A\in\mathcal F$, let
\[
 P_i(A)
 =\{a\in A:\deg_{\mathcal F}(A\setminus\{a\})=i\},
 \qquad
 p_i(A)=|P_i(A)|,
 \qquad i\in\{1,2\}.
\]
Thus, $p_i(A)$ counts the $(k-1)$-subsets of $A$
of degree $i$.
If $a\in P_2(A)$, exactly two members of $\mathcal F$
contain $A\setminus\{a\}$.
One is $A$; write the other as
\[
 (A\setminus\{a\})\cup\{y_A(a)\},
\]
where $y_A(a)\in[n]\setminus A$ is uniquely determined.

\begin{lemma}\label{lem:localrefined}
Let $k\ge4$ and $n>2k$.
Suppose that $\mathcal F\subseteq\binom{[n]}{k}$
contains no $2$-simplex-cluster.
If $A,B\in\mathcal F$ satisfy $A\cap B=\{a,b\}$,
then either
\[
 \{a,b\}\cap P_1(A)\ne\emptyset,
\]
or
\[
 a,b\in P_2(A)
 \qquad\text{and}\qquad
 y_A(a)=y_A(b).
\]
\end{lemma}

\begin{proof}
Suppose that $a,b\notin P_1(A)$.
Define
\[
 \begin{aligned}
 Y_a&=\{y\in[n]\setminus A:
       (A\setminus\{a\})\cup\{y\}\in\mathcal F\},\\
 Y_b&=\{y\in[n]\setminus A:
       (A\setminus\{b\})\cup\{y\}\in\mathcal F\}.
 \end{aligned}
\]
Since $A\in\mathcal F$, both $A\setminus\{a\}$
and $A\setminus\{b\}$ have positive degree.
The assumption $a,b\notin P_1(A)$ excludes degree one.
Hence,
\[
 |Y_a|=\deg_{\mathcal F}(A\setminus\{a\})-1\ge1,\qquad
 |Y_b|=\deg_{\mathcal F}(A\setminus\{b\})-1\ge1.
\]

Assume that there are distinct points
$y_a\in Y_a$ and $y_b\in Y_b$.
Put
\[
 C_a=(A\setminus\{a\})\cup\{y_a\},
 \qquad
 C_b=(A\setminus\{b\})\cup\{y_b\}.
\]
These sets belong to $\mathcal F$, and
\[
 b\in B\cap C_a,\qquad
 a\in B\cap C_b,\qquad
 C_a\cap C_b=A\setminus\{a,b\}.
\]
Since $k\ge4$, the set $A\setminus\{a,b\}$ is nonempty.
Thus, $B,C_a,C_b$ are pairwise intersecting.
They are also distinct, since
$
 a\notin C_a, b\notin C_b, \{a,b\}\subseteq B.
$
On the other hand,
\[
 B\cap C_a\cap C_b
 =B\cap(A\setminus\{a,b\})
 =\emptyset.
\]
Moreover,
\[
 |B\cup C_a\cup C_b|
 \le |A\cup B|+2
 =2k.
\]
Hence, $B,C_a,C_b$ form a $2$-simplex-cluster,
contrary to the hypothesis.

It follows that $y_a=y_b$ for every choice
of $y_a\in Y_a$ and $y_b\in Y_b$.
Since $Y_a$ and $Y_b$ are nonempty, we obtain
$
 Y_a=Y_b=\{y\}
$
for some $y\in[n]\setminus A.$
Therefore, $a,b\in P_2(A)$ and $y_A(a)=y_A(b)=y$.
\end{proof}

The following lemma is a consequence of the equality
case in Bollob\'as's set-pairs theorem~\cite{Bollobas}.
For completeness, we include the direct proof given
by Huang~\cite[Lemma~3.5]{Huang}, adapted to our notation.

\begin{lemma}\label{lem:rigidity}
Let $2\le k<n$ and
$\mathcal F\subseteq\binom{[n]}{k}$.
Suppose that
\[
 |\mathcal F|=\binom{n-1}{k-1},
\]
and that every member of $\mathcal F$
has exactly one $(k-1)$-subset of degree one.
Then $\mathcal F$ is a full star.
\end{lemma}

\begin{proof}
Put
\[
 q=n-k,\qquad M=\binom{n-1}{k-1}.
\]
By hypothesis, each member of $\mathcal F$ contains
exactly one $(k-1)$-subset of degree one.
Each such subset belongs to exactly one member of
$\mathcal F$, so there are exactly $M$ such subsets.

We count pairs $(T,C)$ satisfying
\[
 T\in\binom{[n]}{k-1},\qquad
 C\in\binom{[n]}{k}\setminus\mathcal F,\qquad
 T\subseteq C.
\]
For each $(k-1)$-set $T$, exactly
$q+1-\deg_{\mathcal F}(T)$ of the $k$-sets containing $T$
lie outside $\mathcal F$.
On the other hand, each $k$-set outside $\mathcal F$
contains exactly $k$ $(k-1)$-subsets.
Counting the pairs $(T,C)$ in these two ways gives
\[
 \sum_{T\in\binom{[n]}{k-1}}
 \bigl(q+1-\deg_{\mathcal F}(T)\bigr)
 =k\left(\binom{n}{k}-M\right)=qM,
\]
where the last equality follows from
$\binom{n}{k}=\frac{n}{k}M$ and $q=n-k$.
The $M$ subsets of degree one each contribute $q$
to this sum.
Their total contribution is  $qM$, giving
\[
 \sum_{T\in\binom{[n]}{k-1},
                 \deg_{\mathcal F}(T)\ne1}
 \bigl(q+1-\deg_{\mathcal F}(T)\bigr)=0.
\]
Every summand is nonnegative.
Thus, every $(k-1)$-set whose degree is not one
has degree $q+1$.
Consequently,
\[
 \deg_{\mathcal F}(T)\in\{1,q+1\}
 \qquad\text{for every }T\in\binom{[n]}{k-1}.
\]

Fix $A\in\mathcal F$, and let $x\in A$ be the unique
point such that
$
 \deg_{\mathcal F}(A\setminus\{x\})=1.
$
We show that every $k$-set containing $x$
belongs to $\mathcal F$.

Let $D\in\binom{[n]}{k}$ contain $x$.
Since $|A|=|D|$, we may write
\[
 A\setminus D=\{a_1,\ldots,a_h\},
 \qquad
 D\setminus A=\{y_1,\ldots,y_h\}.
\]
For $0\le i\le h$, put
\[
 A_i=
 \bigl(A\setminus\{a_1,\ldots,a_i\}\bigr)
 \cup\{y_1,\ldots,y_i\}.
\]
Thus, $A_0=A$ and $A_h=D$.
Since $x\in A\cap D$, none of the points $a_i$ equals $x$,
and every $A_i$ contains $x$.
We prove by induction on $i$ that
\[
 A_i\in\mathcal F
 \qquad\text{and}\qquad
 \deg_{\mathcal F}(A_i\setminus\{x\})=1.
\]
The assertion holds for $i=0$ by the choice of $A$ and $x$.
Suppose it holds for $i-1$, where $1\le i\le h$.
The set $A_{i-1}$ has exactly one $(k-1)$-subset
of degree one, namely $A_{i-1}\setminus\{x\}$.
Since $a_i\ne x$, its subset
$A_{i-1}\setminus\{a_i\}$ does not have degree one.
Since every $(k-1)$-set has degree either $1$ or $q+1$,
we have
\[
 \deg_{\mathcal F}(A_{i-1}\setminus\{a_i\})=q+1.
\]
Thus, all $k$-sets containing
$A_{i-1}\setminus\{a_i\}$ belong to $\mathcal F$.
In particular,
\[
 A_i=(A_{i-1}\setminus\{a_i\})\cup\{y_i\}
 \in\mathcal F.
\]
To determine the degree of $A_i\setminus\{x\}$, put
\[
 C_i=(A_{i-1}\setminus\{x\})\cup\{y_i\}.
\]
The point $y_i$ lies outside $A_{i-1}$,
so $C_i$ is a $k$-set distinct from $A_{i-1}$.
By the induction hypothesis,
$A_{i-1}\setminus\{x\}$ has degree one.
Thus, $A_{i-1}$ is the only member of $\mathcal F$
containing $A_{i-1}\setminus\{x\}$.
Since $C_i$ also contains this subset and
$C_i\ne A_{i-1}$, we have $C_i\notin\mathcal F$.
Moreover,
\[
 A_i\setminus\{x\}
 =
 (A_{i-1}\setminus\{x,a_i\})\cup\{y_i\}
 \subseteq C_i.
\]
Hence, $A_i\setminus\{x\}$ has an extension in $\mathcal F$,
namely $A_i$, and an extension outside $\mathcal F$,
namely $C_i$.
Consequently,
\[
 1\le\deg_{\mathcal F}(A_i\setminus\{x\})\le q.
\]
Since its degree must be either $1$ or $q+1$, we obtain
$
 \deg_{\mathcal F}(A_i\setminus\{x\})=1.
$
This completes the induction.
In particular, $D=A_h\in\mathcal F$.

As $D$ was an arbitrary $k$-set containing $x$, we have
$
 \Big\{D\in\binom{[n]}{k}:x\in D\Big\}\subseteq\mathcal F.
$
The family on the left has size
$\binom{n-1}{k-1}=|\mathcal F|$.
Therefore, the two families are equal, and $\mathcal F$
is a full star.
\end{proof}

We now apply the preceding lemmas to prove
Theorem~\ref{thm:triangle}.

\begin{proof}[Proof of Theorem~\ref{thm:triangle}]
Put
\[
 q=n-k,\qquad
 m=|\mathcal F|,\qquad
 M=\binom{n-1}{k-1},\qquad
 \Delta=M-m,
\]
and
\[
 K=\binom{k}{2},
 \qquad
 c=\frac{q-1}{2q}.
\]
Since $n>2k$, we have $q>k$.

\textbf{We first bound the weighted sum}
\[
 \mathcal W
 =\sum_{A\in\mathcal F}\bigl(p_1(A)+c\,p_2(A)\bigr).
\]
Let $a_j$ denote the number of sets
$T\in\binom{[n]}{k-1}$ with $\deg_{\mathcal F}(T)=j$,
for $0\le j\le q+1$.
We count pairs $(T,C)$ satisfying
\[
 T\in\binom{[n]}{k-1},\qquad
 C\in\binom{[n]}{k}\setminus\mathcal F,\qquad
 T\subseteq C.
\]
For each $T$ of degree $j$, exactly $q+1-j$ of the
$k$-sets containing $T$ lie outside $\mathcal F$.
Thus, counting by $T$ gives
$\sum_{j=0}^{q+1}(q+1-j)a_j$ pairs.
On the other hand, each $k$-set outside $\mathcal F$
contains exactly $k$ $(k-1)$-subsets, so counting by $C$
gives $k\left(\binom{n}{k}-m\right)$ pairs.
Since all terms in the first sum are nonnegative,
retaining only those with $j=1$ and $j=2$ gives
\[
 qa_1+(q-1)a_2
 \le \sum_{j=0}^{q+1}(q+1-j)a_j=k\left(\binom{n}{k}-m\right)=qM+k\Delta,
\]
where the last equality follows from $m=M-\Delta$,
$\binom{n}{k}=\frac{n}{k}M$, and $q=n-k$.
For $i\in\{1,2\}$, each $(k-1)$-set of degree $i$
contributes $i$ to $\sum_{A\in\mathcal F}p_i(A)$.
Hence,
\[
 \sum_{A\in\mathcal F}p_1(A)=a_1,
 \qquad
 \sum_{A\in\mathcal F}p_2(A)=2a_2.
\]
Dividing the preceding bound by $q$ yields
\begin{equation}\label{eq:weighted}
 \mathcal W\le M+\frac{k}{q}\Delta.
\end{equation}

\textbf{For each $D\in\binom{[n]}{2}$, define}
\[
 \mathcal H_D
 =\{A\setminus D:A\in\mathcal F,\ D\subseteq A\}
 \subseteq\binom{[n]\setminus D}{k-2}.
\]
Let $\mathcal H_D^*$ be the subfamily of $\mathcal H_D$
consisting of those sets $A\setminus D$ for which either
\[
 D\cap P_1(A)\ne\emptyset,
\]
or
\[
 D=\{a,b\}\subseteq P_2(A)
 \qquad\text{and}\qquad
 y_A(a)=y_A(b).
\]
Finally, set
\[
 \mathcal G_D=\mathcal H_D\setminus\mathcal H_D^*.
\]

Suppose that $A\setminus D$ and $B\setminus D$
are disjoint members of $\mathcal H_D$.
Then $A\cap B=D$.
Lemma~\ref{lem:localrefined}, applied to $(A,B)$ and $(B,A)$,
shows that both $A\setminus D$ and $B\setminus D$
belong to $\mathcal H_D^*$.
Since every disjoint pair in $\mathcal H_D$
has both members in $\mathcal H_D^*$,
each member of
$\mathcal G_D=\mathcal H_D\setminus\mathcal H_D^*$
intersects every member of $\mathcal H_D$,
so the two families are cross-intersecting.

For $A\in\mathcal F$ and $y\in[n]\setminus A$, define
\[
 c_y(A)=|\{a\in P_2(A):y_A(a)=y\}|.
\]
Since each $a\in P_2(A)$ has a unique associated point
$y_A(a)\notin A$, we have
\[
 \sum_{y\notin A}c_y(A)=p_2(A).
\]
Let
\[
 h_A=
 \left|\left\{
 D\in\binom{A}{2}:A\setminus D\in\mathcal H_D^*
 \right\}\right|.
\]
We count these $2$-element subsets $D$ according to
the two conditions in the definition of $\mathcal H_D^*$.

First, consider those satisfying $D\cap P_1(A)\ne\emptyset$.
For each $a\in P_1(A)$, there are exactly $k-1$ choices
of $b\in A\setminus\{a\}$, giving $k-1$ sets $D=\{a,b\}$.
Summing over $a\in P_1(A)$ counts each set meeting
$P_1(A)$ once, except that each set contained in
$P_1(A)$ is counted twice.
There are $\binom{p_1(A)}{2}$ sets of the latter type, so
\[
 \left|\left\{
 D\in\binom{A}{2}:D\cap P_1(A)\ne\emptyset
 \right\}\right|
 =(k-1)p_1(A)-\binom{p_1(A)}{2}.
\]

Next, we count the sets $D=\{a,b\}\in\binom{P_2(A)}{2}$
satisfying $y_A(a)=y_A(b)$.
For each fixed $y\notin A$, the set
$
 \{a\in P_2(A):y_A(a)=y\}
$
has exactly $c_y(A)$ elements.
Choosing two distinct elements from this set gives
$\binom{c_y(A)}{2}$ sets $D=\{a,b\}$ with
$y_A(a)=y_A(b)=y$.
For every $D=\{a,b\}\in\binom{P_2(A)}{2}$ satisfying
$y_A(a)=y_A(b)$, the common value
$y=y_A(a)=y_A(b)$ is uniquely determined.
Thus, summing over $y\notin A$ counts each qualifying
set $D$ exactly once, giving
\[
 \left|\left\{
 \{a,b\}\in\binom{P_2(A)}{2}:y_A(a)=y_A(b)
 \right\}\right|
 =\sum_{y\notin A}\binom{c_y(A)}{2}.
\]

Since $P_1(A)\cap P_2(A)=\emptyset$, no set
$D\in\binom{P_2(A)}{2}$ meets $P_1(A)$.
By the definition of $\mathcal H_D^*$, the sets counted
by $h_A$ are precisely those meeting $P_1(A)$ and those
contained in $P_2(A)$ whose two elements have the same
associated point.
These two collections are disjoint, so
\begin{equation}\label{eq:marks}
 h_A
 =(k-1)p_1(A)-\binom{p_1(A)}{2}
 +\sum_{y\notin A}\binom{c_y(A)}{2}.
\end{equation}

Let $\mathcal E$ consist of those $A\in\mathcal F$
such that $P_2(A)=A$ and all the points $y_A(a)$,
$a\in A$, are equal.
Put $e=|\mathcal E|$.

\textbf{For every $A\in\mathcal F\setminus\mathcal E$, we claim that}
\begin{equation}\label{eq:pointgap}
 h_A\le(k-1)\bigl(p_1(A)+c\,p_2(A)\bigr),
\end{equation}
with strict inequality if $p_2(A)>0$ or $p_1(A)\ge2$.

By \eqref{eq:marks}, the difference between
the right-hand side and the left-hand side is
\begin{align}
 \binom{p_1(A)}{2}
 +\frac{p_2(A)}{2}
   \left(k-p_2(A)-\frac{k-1}{q}\right)
 +\binom{p_2(A)}{2}
 -\sum_{y\notin A}\binom{c_y(A)}{2}.\label{eq:gap}
\end{align}
Observe that $p_2(A)\leq k$ and
\[
 \binom{p_2(A)}{2}
 -\sum_{y\notin A}\binom{c_y(A)}{2}
 =
 \left|\left\{
 \{a,b\}\in\binom{P_2(A)}{2}:y_A(a)\ne y_A(b)
 \right\}\right|
 \ge0.
\]

If $p_2(A)\le k-1$, then
\[
 k-p_2(A)-\frac{k-1}{q}
 \ge1-\frac{k-1}{q}>0.
\]
Thus, the first two terms in \eqref{eq:gap}
and the sum of its last two terms are nonnegative.
The second term is positive when $p_2(A)>0$,
and the first is positive when $p_1(A)\ge2$.

Now suppose that $p_2(A)=k$ and $A\notin\mathcal E$. Then $P_2(A)=A$.
Partition $A$ according to the values of $y_A(a)$.
There are at least two nonempty parts, since
$P_2(A)=A$ and $A\notin\mathcal E$ imply that
the values $y_A(a)$, $a\in A$, are not all equal.
Choose one part of size $s$, where $1\le s\le k-1$.
Let $S$ be the chosen part, with $|S|=s$.
For every $a\in S$ and $b\in A\setminus S$, the points
$a$ and $b$ lie in different parts, so
$y_A(a)\ne y_A(b)$.
There are exactly $s(k-s)$ sets $D=\{a,b\}$ of this form.
Since
$\binom{k}{2}-\sum_{y\notin A}\binom{c_y(A)}{2}$
counts all $2$-element subsets of $A$ whose elements
have different values of $y_A$, we obtain
\[
 \binom{k}{2}
 -\sum_{y\notin A}\binom{c_y(A)}{2}
 \ge s(k-s)\ge k-1.
\]
Since $p_1(A)=0$, the expression in \eqref{eq:gap}
is at least
$
 k-1-\frac{K}{q}>0,
$
where $K=\binom{k}{2}$.
This proves the claim.

For $A\in\mathcal E$, we have $P_2(A)=A$, and the values
$y_A(a)$, $a\in A$, are all equal.
Every $D=\{a,b\}\in\binom{A}{2}$ satisfies
$D\subseteq P_2(A)$ and $y_A(a)=y_A(b)$,
so $A\setminus D\in\mathcal H_D^*$.
We obtain
\[
 p_1(A)=0,\qquad p_2(A)=k,\qquad
 h_A=\binom{k}{2}=K.
\]
Therefore,
\[
 h_A-(k-1)\bigl(p_1(A)+c\,p_2(A)\bigr)
 =K-(k-1)ck=\frac{K}{q}.
\]

We now sum separately over $\mathcal F\setminus\mathcal E$
and $\mathcal E$.
Using \eqref{eq:pointgap} for $A\in\mathcal F\setminus\mathcal E$
and the preceding identity for $A\in\mathcal E$, we obtain
\begin{equation}\label{eq:totalmarks}
 \begin{aligned}
 \sum_{A\in\mathcal F}h_A
 &=\sum_{A\in\mathcal F\setminus\mathcal E}h_A
   +\sum_{A\in\mathcal E}h_A\\
 &\le (k-1)\sum_{A\in\mathcal F\setminus\mathcal E}
       \bigl(p_1(A)+c\,p_2(A)\bigr)
   +\sum_{A\in\mathcal E}
       \left((k-1)\bigl(p_1(A)+c\,p_2(A)\bigr)
       +\frac{K}{q}\right)\\
 &=(k-1)\sum_{A\in\mathcal F}
       \bigl(p_1(A)+c\,p_2(A)\bigr)
   +\frac{K}{q}|\mathcal E|\\
 &=(k-1)\mathcal W+\frac{K}{q}e.
 \end{aligned}
\end{equation}

\textbf{We next apply Lemma~\ref{lem:markednew} to
$\mathcal G_D$ and $\mathcal H_D$ for each
$D\in\binom{[n]}{2}$.}
Define
\[
 \mathcal T_D
 =\{A\setminus D: A\in\mathcal E,\ D\subseteq A\},
 \qquad
 e_D=|\mathcal T_D|.
\]
The definition of $\mathcal E$ gives
$\mathcal T_D\subseteq\mathcal H_D^*$.

Let $R\in\mathcal T_D$, and write $R=A\setminus D$,
where $A\in\mathcal E$ and $D\subseteq A$.
Let $y$ be the common value of $y_A(a)$ for $a\in A$.
Since $P_2(A)=A$, for each $a\in R$, the only members
of $\mathcal F$ containing $A\setminus\{a\}$ are
\[
 A
 \quad\text{and}\quad
 (A\setminus\{a\})\cup\{y\}.
\]
Both sets contain $D$, since $a\notin D$.
Fix $a\in R$. Since $a\notin D$ and $y\notin A$,
the sets
\[
 R
 \quad\text{and}\quad
 (R\setminus\{a\})\cup\{y\}
\]
both belong to $\mathcal H_D$ and contain
$R\setminus\{a\}$.
Indeed, their unions with $D$ are $A$ and
$(A\setminus\{a\})\cup\{y\}$, respectively,
both of which belong to $\mathcal F$.
To show that there are no other such members,
let $B\in\mathcal H_D$ contain $R\setminus\{a\}$.
By the definition of $\mathcal H_D$,
\[
 B\cup D\in\mathcal F,
 \qquad
 A\setminus\{a\}
 =(R\setminus\{a\})\cup D
 \subseteq B\cup D.
\]
Since $a\in P_2(A)$, we have
$\deg_{\mathcal F}(A\setminus\{a\})=2$.
Hence, $B\cup D$ must be either $A$ or
$(A\setminus\{a\})\cup\{y\}$.
Using $B=(B\cup D)\setminus D$, we obtain
\[
 B=R
 \quad\text{or}\quad
 B=(R\setminus\{a\})\cup\{y\}.
\]
This holds for every $a\in R$, with the same point
$y\notin A=R\cup D$, and verifies the additional
condition in Lemma~\ref{lem:markednew}.

The members of $\mathcal H_D$ are $(k-2)$-subsets
of $[n]\setminus D$.
After relabelling the elements of $[n]\setminus D$,
we may regard $\mathcal H_D$ as a family of
$(k-2)$-subsets of $[n-2]$.
Since $k\ge4$ and $n>2k$, we have
\[
 k-2\ge2,\qquad
 n-2\ge2(k-2)+3,\qquad
 (n-2)-(k-2)=q.
\]
We have already shown that $\mathcal G_D$ and
$\mathcal H_D$ are cross-intersecting and that
$\mathcal T_D$ satisfies the additional condition.
Applying Lemma~\ref{lem:markednew} with $n$ and $k$
replaced by $n-2$ and $k-2$, respectively, gives
\begin{equation}\label{eq:linkcount}
 |\mathcal H_D|
 \le\binom{n-3}{k-3}
   +\frac{q}{n-2}|\mathcal H_D^*|
   -\frac{e_D}{n-2},
\end{equation}
with strict inequality whenever $e_D>0$.

Each $A\in\mathcal F$ contains exactly
$K=\binom{k}{2}$ subsets $D$ of size two.
For each of these choices, $A\setminus D$ contributes
one to $|\mathcal H_D|$.
By definition, exactly $h_A$ of these choices satisfy
$A\setminus D\in\mathcal H_D^*$.
Similarly, each $A\in\mathcal E$ contributes one
to $e_D$ for each of its $K$ subsets $D$ of size two.
Summing over $D\in\binom{[n]}{2}$ gives
\[
 \sum_{D\in\binom{[n]}{2}}|\mathcal H_D|=Km,\qquad
 \sum_{D\in\binom{[n]}{2}}|\mathcal H_D^*|=\sum_{A\in\mathcal F}h_A,\qquad
 \sum_{D\in\binom{[n]}{2}}e_D=Ke.
\]
Summing \eqref{eq:linkcount} over
$D\in\binom{[n]}{2}$, we obtain
\[
 Km
 \le\binom{n}{2}\binom{n-3}{k-3}
   +\frac{q}{n-2}\sum_{A\in\mathcal F}h_A
   -\frac{Ke}{n-2}.
\]
Substituting \eqref{eq:totalmarks} and cancelling
the terms involving $e$ gives
\begin{align}
 Km
 &\le\binom{n}{2}\binom{n-3}{k-3}
   +\frac{q}{n-2}
       \left((k-1)\mathcal W+\frac{K}{q}e\right)
   -\frac{Ke}{n-2}
 \nonumber\\
 &=\binom{n}{2}\binom{n-3}{k-3}
   +\frac{q(k-1)}{n-2}\mathcal W.
 \label{eq:cancellation}
\end{align}

We use the identity
\[
 \binom{n}{2}\binom{n-3}{k-3}
 +\frac{q(k-1)}{n-2}M=KM.
\]
Indeed, dividing by $M$ reduces this identity to
\[
 \frac{n(k-1)(k-2)}{2(n-2)}
 +\frac{q(k-1)}{n-2}
 =\frac{k(k-1)}{2},
\]
which follows from $q=n-k$.
Combining \eqref{eq:cancellation} with \eqref{eq:weighted}
and using the preceding identity, we obtain
\[
 \begin{aligned}
 Km
 &\le\binom{n}{2}\binom{n-3}{k-3}
   +\frac{q(k-1)}{n-2}
      \left(M+\frac{k}{q}\Delta\right)\\
 &=KM+\frac{k(k-1)}{n-2}\Delta=KM+\frac{2K}{n-2}\Delta.
 \end{aligned}
\]
Since $m=M-\Delta$, this is equivalent to
\[
 0\le K\left(1+\frac{2}{n-2}\right)\Delta
   =\frac{nK}{n-2}\Delta.
\]
The coefficient of $\Delta$ is positive, so
$\Delta\ge0$, proving $m\le M$.

\textbf{Finally, suppose that $m=M$.}
Then $\Delta=0$, and the chain of estimates leading
to the bound
\[
 Km\le KM+\frac{2K}{n-2}\Delta
\]
begins and ends with $KM$.
Equality must hold in every inequality used in this chain,
including \eqref{eq:weighted}, \eqref{eq:totalmarks},
and \eqref{eq:linkcount} for every $D\in\binom{[n]}{2}$.

If $e>0$, the identity
$\sum_{D\in\binom{[n]}{2}}e_D=Ke$ gives
$e_D>0$ for some $D\in\binom{[n]}{2}$.
For this $D$, inequality \eqref{eq:linkcount} is strict,
contradicting equality in the summed bound.
Hence, $e=0$ and $\mathcal E=\emptyset$.
With $\mathcal E=\emptyset$, equality in
\eqref{eq:totalmarks} forces equality in
\eqref{eq:pointgap} for every $A\in\mathcal F$.
The strictness assertion following \eqref{eq:pointgap}
gives
\[
 p_2(A)=0,\qquad p_1(A)\le1
 \qquad\text{for every }A\in\mathcal F.
\]
Equality in \eqref{eq:weighted}, together with
$p_2(A)=0$ for every $A\in\mathcal F$, gives
\[
 \sum_{A\in\mathcal F}p_1(A)=\mathcal W=M=m.
\]
The $m$ summands are at most one and sum to $m$,
so $p_1(A)=1$ for every $A\in\mathcal F$.
Lemma~\ref{lem:rigidity} now shows that
$\mathcal F$ is a full star.

Conversely, a full star has size $M$.
Every three of its members have a common point,
so it contains no $2$-simplex-cluster.

\end{proof}

\section{Completion of the proof}\label{sec:completion}

\subsection{Disjoint pairs in the Kneser graph}

We begin with an elementary observation that will be used
to characterize equality in the spectral estimate below.

\begin{lemma}\label{lem:boolean}
Let $n,r$ be integers with $2\le r\le n-2$, and let
$f:\binom{[n]}{r}\longrightarrow\{0,1\}$.
Suppose that there exist real numbers $a,b_1,\ldots,b_n$
such that
\[
 f(R)=a+\sum_{i\in R}b_i
 \qquad\text{for every }R\in\binom{[n]}{r}.
\]
Then $f$ is constant, or there exists $x\in[n]$ such that
one of the identities
\[
 f(R)=\mathbf1_{x\in R}
 =
 \begin{cases}
  1, & x\in R,\\
  0, & x\notin R,
 \end{cases}
 \qquad\text{or}\qquad
 f(R)=\mathbf1_{x\notin R}=1-\mathbf1_{x\in R}
\]
holds for every $R\in\binom{[n]}{r}$.
\end{lemma}

\begin{proof}
For distinct $i,j\in[n]$, choose
$T\in\binom{[n]\setminus\{i,j\}}{r-1}$.
The assumed representation of $f$ gives
\[
 b_i-b_j
 =f(T\cup\{i\})-f(T\cup\{j\})
 \in\{-1,0,1\}.
\]
If $b_1=\cdots=b_n=\beta$, then
$f(R)=a+r\beta$ for every $R$, so $f$ is constant.

Suppose that the coefficients are not all equal, and put
$\beta=\min_{i\in[n]}b_i$.
Choose $j\in[n]$ such that $b_j=\beta$.
For every $i\in[n]$, the preceding identity gives
\[
 b_i-\beta=b_i-b_j\in\{-1,0,1\}.
\]
Since $\beta$ is the smallest coefficient, we also have
$b_i-\beta\ge0$, so $b_i\in\{\beta,\beta+1\}$.
The value $\beta$ occurs at $j$.
As the coefficients are not all equal, some coefficient
is greater than $\beta$ and must equal $\beta+1$.
Thus, the coefficients take exactly these two values.
Define
\[
 P=\{i\in[n]:b_i=\beta+1\},
 \qquad
 a'=a+r\beta.
\]
Then $1\le|P|\le n-1$, and
\[
 f(R)=a'+|R\cap P|
 \qquad\text{for every }R\in\binom{[n]}{r}.
\]

We claim that $|P|=1$ or $n-1$.
Suppose instead that $2\le|P|\le n-2$.
Choose distinct $u,v\in P$ and distinct
$w,z\in[n]\setminus P$.
Since $0\le r-2\le n-4$, we can choose
$
 S\in\binom{[n]\setminus\{u,v,w,z\}}{r-2}.
$
The three $r$-sets obtained by adding
$\{w,z\}$, $\{u,w\}$, and $\{u,v\}$ to $S$ satisfy
\[
 \begin{aligned}
 f(S\cup\{w,z\})&=a'+|S\cap P|,\\
 f(S\cup\{u,w\})&=a'+|S\cap P|+1,\\
 f(S\cup\{u,v\})&=a'+|S\cap P|+2.
 \end{aligned}
\]
These are three distinct values, contradicting
$f(R)\in\{0,1\}$.
This proves the claim.

If $P=\{x\}$, then
\[
 f(R)=a'+\mathbf1_{x\in R}.
\]
There are $r$-sets containing $x$ and $r$-sets avoiding $x$,
so both $a'$ and $a'+1$ occur as values of $f$.
Since both belong to $\{0,1\}$, we have $a'=0$.
Thus, $f(R)=\mathbf1_{x\in R}$ for every $R$.

If $P=[n]\setminus\{x\}$, then
\[
 f(R)=a'+r-\mathbf1_{x\in R}.
\]
The two values $a'+r-1$ and $a'+r$ both occur
and belong to $\{0,1\}$, so $a'+r=1$.
It follows that
\[
 f(R)=1-\mathbf1_{x\in R}
     =\mathbf1_{x\notin R}
 \qquad\text{for every }R\in\binom{[n]}{r}.
\]
This completes the proof.
\end{proof}

For integers $n\ge r\ge1$, the \textit{Kneser graph} $KG(n,r)$
has vertex set $\binom{[n]}{r}$, with two vertices adjacent if and only if
they are disjoint.

\begin{lemma}[{\cite[Corollary~6.6.1]{GodsilMeagher}}]
\label{lem:kneserspectrum}
Let $n,r$ be positive integers with $n>2r$.
The eigenvalues of the adjacency matrix of $KG(n,r)$ are
\[
 \lambda_j=(-1)^j\binom{n-r-j}{r-j},
 \qquad j=0,\ldots,r.
\]
The multiplicity of $\lambda_j$ is
$\binom{n}{j}-\binom{n}{j-1}$, where
$\binom{n}{-1}=0$.
In particular, the least eigenvalue is
\[
 \lambda_1=-\binom{n-r-1}{r-1},
\]
with multiplicity $n-1$.
\end{lemma}

We use Lemma~\ref{lem:kneserspectrum} to obtain the
following lower bound on the number of ordered disjoint
pairs in an $r$-uniform family.

\begin{lemma}\label{lem:kneser}
Let $n,r$ be positive integers with $n>2r$, and put
$k=n-r$, $N=\binom{n}{r}$, and $D=\binom{k}{r}$.
For $\mathcal G\subseteq\binom{[n]}{r}$, put
$m=|\mathcal G|$ and $\alpha=m/N$, and define
$
 I=\left|\left\{
 (R,S)\in\mathcal G\times\mathcal G:
 R\cap S=\emptyset
 \right\}\right|.
$
Then
\begin{equation}\label{eq:kneser}
 I\ge mD\left(\frac nk\alpha-\frac rk\right).
\end{equation}
If equality holds, then the indicator function
$f=\mathbf1_{\mathcal G}$, defined by $f(R)=1$ when
$R\in\mathcal G$ and $f(R)=0$ otherwise, can be written as
\[
 f(R)=a+\sum_{i\in R}b_i
 \qquad\text{for every }R\in\binom{[n]}{r},
\]
for some real numbers $a,b_1,\ldots,b_n$.
\end{lemma}

\begin{proof}
Let $\mathsf A$ be the adjacency matrix of $KG(n,r)$.
For real-valued functions $u,v$ on $\binom{[n]}{r}$,
use the inner product and norm
\[
 \langle u,v\rangle
 =\sum_{R\in\binom{[n]}{r}}u(R)v(R),
 \qquad
 \|u\|^2=\langle u,u\rangle.
\]
The matrix $\mathsf A$ acts on these functions by
\[
 (\mathsf A u)(R)
 =\sum_{S\in\binom{[n]}{r},
                  S\cap R=\emptyset}u(S).
\]
Write $\mathbf1$ for the constant function with value one.
Each vertex has exactly $D$ neighbors, so
$\mathsf A\mathbf1=D\mathbf1$.
By Lemma~\ref{lem:kneserspectrum}, the least eigenvalue
of $\mathsf A$ is
\[
 \lambda=-\binom{k-1}{r-1}=-\frac rkD,
\]
with multiplicity $n-1$.

For the indicator function $f=\mathbf1_{\mathcal G}$,
put $h=f-\alpha\mathbf1$.
Since $f$ takes the value one on exactly $m$ vertices
and zero on the remaining vertices, we have
\[
 \langle h,\mathbf1\rangle
 =m-N\alpha=0,\qquad
 \|h\|^2
 =m-2\alpha m+N\alpha^2
  =N\alpha(1-\alpha).
\]
For each $R\in\binom{[n]}{r}$, $(\mathsf A f)(R)$
is the number of members of $\mathcal G$ disjoint from $R$.
Since $f(R)=1$ when $R\in\mathcal G$ and zero otherwise,
\[
 \langle f,\mathsf A f\rangle
 =\sum_{R\in\mathcal G}(\mathsf A f)(R)=\sum_{R\in\mathcal G}
   \left|\{S\in\mathcal G:R\cap S=\emptyset\}\right|
 =I.
\]
Expanding $f=\alpha\mathbf1+h$ and using the symmetry
of $\mathsf A$, we obtain
\[
 \begin{aligned}
 I
 &=\langle\alpha\mathbf1+h,
          \mathsf A(\alpha\mathbf1+h)\rangle=\alpha^2\langle\mathbf1,\mathsf A\mathbf1\rangle
   +2\alpha\langle h,\mathsf A\mathbf1\rangle
   +\langle h,\mathsf A h\rangle\\
 &=ND\alpha^2+2D\alpha\langle h,\mathbf1\rangle
   +\langle h,\mathsf A h\rangle\\
&=ND\alpha^2+\langle h,\mathsf A h\rangle,
 \end{aligned}
\]
where the last two equalities use
$\mathsf A\mathbf1=D\mathbf1$,
$\langle\mathbf1,\mathbf1\rangle=N$, and
$\langle h,\mathbf1\rangle=0$.
Since $\lambda$ is the least eigenvalue of the symmetric
matrix $\mathsf A$, we have
$\langle h,\mathsf A h\rangle\ge\lambda\|h\|^2$.
Substituting $\lambda=-\frac rkD$ and
$\|h\|^2=N\alpha(1-\alpha)$ gives
\[
 \begin{aligned}
 I
 &\ge ND\alpha^2+\lambda\|h\|^2=ND\alpha^2-\frac rkD\,N\alpha(1-\alpha)\\
 &=ND\alpha
   \left(\left(1+\frac rk\right)\alpha-\frac rk\right)=mD\left(\frac nk\alpha-\frac rk\right).
 \end{aligned}
\]

Assume now that equality holds.
Let $E$ be the eigenspace of $\mathsf A$
corresponding to $\lambda$.
Equality in $\langle h,\mathsf A h\rangle\ge\lambda\|h\|^2$
implies $h\in E$.
Since $\mathsf A$ is symmetric and $\lambda$ has
multiplicity $n-1$, we have $\dim E=n-1$.
For each $i\in[n]$, define
\[
 u_i(R)=\mathbf1_{i\in R}-\frac rn,
\]
where $\mathbf1_{i\in R}$ equals one when $i\in R$
and zero otherwise.
If $i\notin R$, exactly $\binom{k-1}{r-1}$
$r$-sets disjoint from $R$ contain $i$;
if $i\in R$, no $r$-set disjoint from $R$ contains $i$.
Since exactly $D$ members of $\binom{[n]}{r}$ are
disjoint from $R$, we obtain
\[
 \begin{aligned}
 (\mathsf A u_i)(R)
 &=\sum_{S\in\binom{[n]}{r},
                   S\cap R=\emptyset}
      \left(\mathbf1_{i\in S}-\frac rn\right)=\binom{k-1}{r-1}(1-\mathbf1_{i\in R})-\frac rnD\\
 &=-\binom{k-1}{r-1}
      \left(\mathbf1_{i\in R}-\frac rn\right)=\lambda u_i(R).
 \end{aligned}
\]
Here the third equality uses
$\binom{k-1}{r-1}=\frac rkD$ and $k=n-r$.
Thus, $u_i\in E$ for every $i\in[n]$.

For every $R\in\binom{[n]}{r}$, we have
\[
 \sum_{i=1}^n u_i(R)
 =\sum_{i=1}^n\mathbf1_{i\in R}-r
 =|R|-r=0.
\]
Hence, $\sum_{i=1}^n u_i$ is the zero function.
Suppose that real numbers $c_1,\ldots,c_{n-1}$ satisfy
$\sum_{\ell=1}^{n-1}c_\ell u_\ell=0$.
For each $i\in[n-1]$, choose
$T\in\binom{[n]\setminus\{i,n\}}{r-1}$.
Evaluating the relation at $T\cup\{i\}$ and
$T\cup\{n\}$ and subtracting gives
\[
 0=\sum_{\ell=1}^{n-1}c_\ell
 \bigl(u_\ell(T\cup\{i\})-u_\ell(T\cup\{n\})\bigr)
 =c_i.
\]
Thus, $c_1=\cdots=c_{n-1}=0$, proving that
$u_1,\ldots,u_{n-1}$ are linearly independent.
Since these functions belong to $E$ and $\dim E=n-1$,
they form a basis of $E$.

As $h\in E$, there exist real numbers $b_1,\ldots,b_n$
such that $h=\sum_{i=1}^n b_i u_i$.
Put $a=\alpha-\frac rn\sum_{i=1}^n b_i$.
Using $f=\alpha\mathbf1+h$, $h=\sum_{i=1}^n b_i u_i$,
and $a=\alpha-\frac rn\sum_{i=1}^n b_i$, we obtain,
for every $R\in\binom{[n]}{r}$,
\[
 f(R)
 =\alpha+\sum_{i=1}^n b_i
      \left(\mathbf1_{i\in R}-\frac rn\right)=\left(\alpha-\frac rn\sum_{i=1}^n b_i\right)
   +\sum_{i\in R}b_i=a+\sum_{i\in R}b_i.
\]
This proves the equality assertion.
\end{proof}

\subsection{Local counting}

The following estimate will be used when $n=2k-d+1$.
Its proof adapts Huang's argument%
~\cite[Lemma~3.6]{Huang} to families containing no
$d$-simplex-cluster, using Currier's local counting
estimates~\cite[Lemmas~1--3]{Currier}.

\begin{lemma}\label{lem:transition}
Let $d\ge3$, $k\ge d+1$, $n\ge2k-d$, and $r=n-k$.
Suppose that $d(r-1)>2r$.
If $\mathcal F\subseteq\binom{[n]}{k}$ contains no
$d$-simplex-cluster, then
\[
 |\mathcal F|\le\binom{n-1}{k-1},
\]
with equality only for a full star.
\end{lemma}

\begin{proof}
Put $m=|\mathcal F|$ and $M=\binom{n-1}{k-1}$.
If $m<M$, the desired bound holds strictly.
It remains to consider the case $m\ge M$.
Use the notation $\deg_{\mathcal F}$, $P_i(A)$, and
$p_i(A)$ from Section~\ref{sec:triangle}.

For $j\in\{0,\ldots,r+1\}$, let $a_j$ be the number
of $(k-1)$-sets of degree $j$.
We count pairs $(T,C)$ satisfying
\[
 T\in\binom{[n]}{k-1},\qquad
 C\in\binom{[n]}{k}\setminus\mathcal F,\qquad
 T\subseteq C.
\]
A set $T$ of degree $j$ is contained in exactly
$r+1-j$ choices of $C$, while each such $C$ contains
exactly $k$ choices of $T$.
Thus,
\[
 \begin{aligned}
 r a_1+(r-1)a_2
 &\le\sum_{j=0}^{r+1}(r+1-j)a_j=k\left(\binom{n}{k}-m\right)\\
 &\le k\left(\binom{n}{k}-M\right)=rM.
 \end{aligned}
\]
For $i\in\{1,2\}$, each $(k-1)$-set of degree $i$
contributes $i$ to $\sum_{A\in\mathcal F}p_i(A)$.
Hence,
\[
 \sum_{A\in\mathcal F}p_1(A)=a_1,\qquad
 \sum_{A\in\mathcal F}p_2(A)=2a_2.
\]
It follows that
\begin{equation}\label{eq:transitionweight}
 \mathcal W
 :=\sum_{A\in\mathcal F}
   \left(p_1(A)+\frac{r-1}{2r}p_2(A)\right)=\frac{r a_1+(r-1)a_2}{r}\le M.
\end{equation}

For $D\in\binom{[n]}{d}$, define
\[
 \mathcal H_D
 =\{A\setminus D:A\in\mathcal F,\ D\subseteq A\}
 \subseteq\binom{[n]\setminus D}{k-d}.
\]
Let $\mathcal H_D^*$ consist of the members
$A\setminus D$ for which
\[
 D\cap P_1(A)\ne\emptyset
 \qquad\text{or}\qquad
 D\subseteq P_2(A).
\]
We show that both members of every disjoint pair
in $\mathcal H_D$ belong to $\mathcal H_D^*$.

Suppose that $A\setminus D$ and $B\setminus D$
are disjoint members of $\mathcal H_D$, so
$A\cap B=D$.
Write $D=\{x_1,\ldots,x_d\}$ and suppose that
$A\setminus D\notin\mathcal H_D^*$.
Then every $A\setminus\{x_i\}$ has degree at least two,
and at least one has degree at least three.
After relabeling, assume that
$A\setminus\{x_1\}$ has degree at least three.
We can choose
\[
 A_i=(A\setminus\{x_i\})\cup\{y_i\}\in\mathcal F,
 \qquad y_i\in[n]\setminus A,\quad i\in[d],
\]
with $y_1\ne y_2$.
Indeed, choose $y_2,\ldots,y_d$ first.
Since $\deg_{\mathcal F}(A\setminus\{x_1\})\ge3$,
there are at least two extensions other than $A$,
so $y_1$ can be chosen different from $y_2$.
The sets $B,A_1,\ldots,A_d$ are distinct and satisfy
\[
 \begin{aligned}
 \bigcap_{i=1}^d A_i=A\setminus D\ne\emptyset,\qquad
 B\cap\bigcap_{i=1}^d A_i=\emptyset,\qquad
 x_i\in B\cap\bigcap_{j\in[d]\setminus\{i\}}A_j
 \qquad(i\in[d]).
 \end{aligned}
\]
Moreover,
\[
 \left|B\cup\bigcup_{i=1}^d A_i\right|
 \le |A\cup B|+d=2k.
\]
They form a $d$-simplex-cluster, a contradiction.
Thus, $A\setminus D\in\mathcal H_D^*$.
Interchanging $A$ and $B$ gives
$B\setminus D\in\mathcal H_D^*$.

Every member of $\mathcal H_D\setminus\mathcal H_D^*$
now intersects every member of $\mathcal H_D$.
Since
\[
 k-d\ge1,\qquad
 n-d\ge2(k-d),\qquad
 (n-d)-(k-d)=r,
\]
we may apply Lemma~\ref{lem:markedbasic} to these
two families, after relabeling $[n]\setminus D$ as
$[n-d]$.
This gives
\begin{equation}\label{eq:transitionlink}
 |\mathcal H_D|
 \le\binom{n-d-1}{k-d-1}
   +\frac r{n-d}|\mathcal H_D^*|.
\end{equation}

Put $C=\binom{k-1}{d-1}$.
For $A\in\mathcal F$, let $h(A)$ be the number of
$d$-element subsets $D\subseteq A$ for which
$A\setminus D\in\mathcal H_D^*$.
Exactly
$
 \binom{k}{d}-\binom{k-p_1(A)}{d}
$
of these subsets meet $P_1(A)$, and exactly
$\binom{p_2(A)}{d}$ are contained in $P_2(A)$.
The two conditions are disjoint because
$P_1(A)\cap P_2(A)=\emptyset$.
By the definition of $\mathcal H_D^*$, we have
\[
 h(A)
 =\binom{k}{d}-\binom{k-p_1(A)}{d}
   +\binom{p_2(A)}{d}.
\]
For $0\le p\le k$, we have
\[
 \binom{k}{d}-\binom{k-p}{d}
 =\sum_{j=1}^{p}\binom{k-j}{d-1}\le pC.
\]
For $1\le p\le k$, we also have
\[
 \binom{p}{d}
 =\frac pd\binom{p-1}{d-1}\le\frac pd C,
\]
and this bound holds trivially when $p=0$.
Consequently,
\[
 h(A)\le C\left(p_1(A)+\frac{p_2(A)}d\right).
\]
The assumption $d(r-1)>2r$ is equivalent to
$1/d<(r-1)/(2r)$.
Summing over $A\in\mathcal F$ and using
\eqref{eq:transitionweight}, we obtain
\begin{equation}\label{eq:transitionmarks}
 \sum_{D\in\binom{[n]}{d}}|\mathcal H_D^*|
 =\sum_{A\in\mathcal F}h(A)
 \le C\mathcal W\le CM.
\end{equation}

Each member of $\mathcal F$ contains exactly
$\binom{k}{d}$ subsets of size $d$, so
$
 \sum_{D\in\binom{[n]}{d}}|\mathcal H_D|
 =\binom{k}{d}m.
$
Summing \eqref{eq:transitionlink} and using
\eqref{eq:transitionmarks}, we obtain
\[
 \begin{aligned}
 \binom{k}{d}m
 &\le\binom{n}{d}\binom{n-d-1}{k-d-1}
   +\frac r{n-d}\sum_{D\in\binom{[n]}{d}}
      |\mathcal H_D^*|\\
 &\le\binom{n}{d}\binom{n-d-1}{k-d-1}
   +\frac r{n-d}CM=\binom{k}{d}M.
 \end{aligned}
\]
Thus, $m\le M$.

If $m=M$, equality holds in all the preceding bounds.
Since $1/d<(r-1)/(2r)$, equality in
\eqref{eq:transitionmarks} forces $p_2(A)=0$
for every $A\in\mathcal F$.
The bound
\[
 \binom{k}{d}-\binom{k-p_1(A)}{d}\le Cp_1(A)
\]
is strict when $p_1(A)\ge2$, because
$\binom{k-2}{d-1}<C$.
Hence, $p_1(A)\le1$ for every $A\in\mathcal F$.
Equality in \eqref{eq:transitionmarks} also gives
$\mathcal W=M=m$.
It follows that
$
 \sum_{A\in\mathcal F}p_1(A)=m,
$
so $p_1(A)=1$ for every $A\in\mathcal F$.
Lemma~\ref{lem:rigidity} shows that $\mathcal F$
is a full star.
\end{proof}

For $n=2k-d+1$, Lemma~\ref{lem:transition}
applies whenever $d,r\ge3$, except at
$(d,k,n)=(3,5,8)$.
The following two lemmas handle $r=2$ and this
exceptional case.

\begin{lemma}\label{lem:ranktwo}
Let $n\ge5$ and
$\mathcal F\subseteq\binom{[n]}{n-2}$.
If $\mathcal F$ contains no $(n-3)$-simplex, then
\[
 |\mathcal F|\le\binom{n-1}{2},
\]
with equality only for a full star.
\end{lemma}

\begin{proof}
Put
$
 \mathcal G=\{[n]\setminus A:A\in\mathcal F\}
 \subseteq\binom{[n]}{2},
$
and view $\mathcal G$ as the edge set of a graph on $[n]$.
We may assume $|\mathcal G|\ge\binom{n-1}{2}$,
as otherwise the desired bound holds strictly.

If a vertex $x$ is isolated, then
$
 \mathcal F\subseteq
 \left\{A\in\binom{[n]}{n-2}:x\in A\right\}.
$
This gives $|\mathcal F|\le\binom{n-1}{2}$,
with equality only for the full star with centre $x$.
Suppose that the graph has no isolated vertex.
Its average degree satisfies
\[
 \frac{2|\mathcal G|}{n}
 \ge\frac{(n-1)(n-2)}n
 =n-3+\frac2n>n-3.
\]
Thus, some vertex $v$ has degree at least $n-2$.
If $v$ has a nonneighbor $z$, then $z$ is not isolated,
so we can choose an edge $zw\in\mathcal G$.
Since $vz\notin\mathcal G$, we have $w\ne v$.
If $v$ is adjacent to every other vertex, exactly
$n-1$ edges contain $v$.
Since
$
 |\mathcal G|\ge\binom{n-1}{2}>n-1,
$
we can choose an edge $zw\in\mathcal G$
not containing $v$.

Put $W=[n]\setminus\{v,z,w\}$.
Since $v$ has degree at least $n-2$, its only possible
nonneighbor is $z$.
As $z\notin W$, $v$ is adjacent to every point of $W$.
Hence,
$
 A_0=[n]\setminus\{z,w\},\
 A_x=[n]\setminus\{v,x\},\ x\in W,
$
are $n-2$ distinct members of $\mathcal F$.
They satisfy
\[
 A_0\cap\bigcap_{x\in W}A_x=\emptyset,\qquad
 \bigcap_{x\in W}A_x=\{z,w\},\qquad
 A_0\cap\bigcap_{t\in W\setminus\{x\}}A_t
 =\{x\}\qquad(x\in W).
\]
Thus, they form an $(n-3)$-simplex, a contradiction.
\end{proof}

\begin{lemma}\label{lem:eight}
If $\mathcal F\subseteq\binom{[8]}{5}$ contains no
$3$-simplex, then $|\mathcal F|\le35$,
with equality only for a full star.
\end{lemma}

\begin{proof}
Put
$
 \mathcal G=\{[8]\setminus A:A\in\mathcal F\}
 \subseteq\binom{[8]}{3}.
$
Fix a cyclic ordering $(x_0,\ldots,x_7)$ of $[8]$, and put
$
 I_i=\{x_i,x_{i+1},x_{i+2}\},
$
with indices read modulo $8$.
The eight intervals are partitioned into
\[
 \{I_0,I_4\},\quad
 \{I_1,I_5\},\quad
 \{I_2,I_6\},\quad
 \{I_3,I_7\}.
\]

Suppose that at least six of the eight intervals
belong to $\mathcal G$.
Then
$
 \sum_{i=0}^3
 \left|\mathcal G\cap\{I_i,I_{i+4}\}\right|
 \ge6>2+1+1+1.
$
Each summand is at most two, so at least two
summands must equal two.
Hence, there exist $0\le i<j\le3$ such that
$
 I_i, I_{i+4}, I_j, I_{j+4}\in\mathcal G.
$
Renumber the cyclic positions so that position $i$
becomes $0$, and put $h=j-i$.
The four intervals are then
\[
 I_0,\quad I_h,\quad I_4,\quad I_{h+4},
 \qquad 1\le h\le3.
\]
Since $1\le h\le3$, we have
$
 I_0\cup I_4=[8]\setminus\{x_3,x_7\},
  x_3\in I_h, x_7\in I_{h+4}.
$
Hence, $I_0\cup I_h\cup I_4\cup I_{h+4}=[8]$.
Put $A_i=[8]\setminus I_i$ for
$i\in\{0,h,4,h+4\}$.
These four sets belong to $\mathcal F$ and satisfy
\[
 A_0\cap A_h\cap A_4\cap A_{h+4}
 =[8]\setminus(I_0\cup I_h\cup I_4\cup I_{h+4})
 =\emptyset.
\]
The definitions of the intervals also give
\[
 \begin{aligned}
 A_h\cap A_4\cap A_{h+4}&=\{x_{h-1}\},
 &\qquad A_0\cap A_4\cap A_{h+4}&=\{x_3\},\\
 A_0\cap A_h\cap A_{h+4}&=\{x_{h+3}\},
 &\qquad A_0\cap A_h\cap A_4&=\{x_7\}.
 \end{aligned}
\]
Thus, $A_0,A_h,A_4,A_{h+4}$  form a $3$-simplex in $\mathcal F$,
a contradiction.

Therefore, every cyclic ordering has at most five intervals
in $\mathcal G$.
Let $\Omega$ be the set of cyclic orderings of $[8]$.
Then $|\Omega|=7!$.
Each $R\in\mathcal G$ occurs as an interval
in exactly $3!5!$ cyclic orderings.
Double counting gives
$
 3!5!\,|\mathcal G|
=\sum_{\sigma\in\Omega}
   \left|\{R\in\mathcal G:
   R\text{ is an interval in }\sigma\}\right|\le5|\Omega|=5\cdot7!.
$
Hence,
\[
 |\mathcal F|=|\mathcal G|
 \le\frac{5\cdot7!}{3!5!}=35.
\]

Suppose that equality holds.
Every cyclic ordering then has exactly five intervals
in $\mathcal G$.
For distinct $x,y\in[8]$, let $S=\{s_1,s_2\}$ and
$T=\{t_1,t_2\}$ be disjoint members of
$\binom{[8]\setminus\{x,y\}}{2}$.
Consider the cyclic ordering
\[
 (s_1,s_2,x,y,t_1,t_2,z_1,z_2),
\]
where $z_1,z_2$ are the remaining two points.
Exchanging $x$ and $y$ replaces $S\cup\{x\}$ and
$T\cup\{y\}$ by $S\cup\{y\}$ and $T\cup\{x\}$,
respectively. The other six three-element sets are
common to both orderings.
Let $t$ be the number of these six sets belonging
to $\mathcal G$, and let $g=\mathbf1_{\mathcal G}$.
Since each ordering has exactly five intervals
in $\mathcal G$, we obtain
\[
 g(S\cup\{x\})+g(T\cup\{y\})
 =5-t=g(S\cup\{y\})+g(T\cup\{x\}).
\]

Let $S,T\in\binom{[8]\setminus\{x,y\}}{2}$.
Since $|\{x,y\}\cup S\cup T|\le6$, we can choose
$
 U\in\binom{[8]\setminus(\{x,y\}\cup S\cup T)}{2}.
$
Applying the preceding identity to $(S,U)$ and $(T,U)$ gives
\[
 g(S\cup\{x\})-g(S\cup\{y\})
 =g(U\cup\{x\})-g(U\cup\{y\})=g(T\cup\{x\})-g(T\cup\{y\}).
\]
Thus, $g(S\cup\{x\})-g(S\cup\{y\})$ is independent
of $S$. Denote its common value by $c_{xy}$.
Put $b_8=0$ and $b_x=c_{x8}$ for $x\in[7]$.
For distinct $x,y\in[7]$, choose
$S\in\binom{[8]\setminus\{x,y,8\}}{2}$.
Then
\[
g(S\cup\{x\})-g(S\cup\{y\})
 =\bigl(g(S\cup\{x\})-g(S\cup\{8\})\bigr)
   -\bigl(g(S\cup\{y\})-g(S\cup\{8\})\bigr).
\]
It follows that $c_{xy}=b_x-b_y$ for distinct $x,y\in[7]$.
For $x\in[7]$, the definitions also give
$
 c_{x8}=b_x=b_x-b_8,
 c_{8x}=-c_{x8}=b_8-b_x.
$
Thus, $c_{xy}=b_x-b_y$ for all distinct $x,y\in[8]$.
Let $R\in\binom{[8]}{3}$, $x\in R$, and
$y\in[8]\setminus R$.
Put $R'=(R\setminus\{x\})\cup\{y\}$.
Taking $S=R\setminus\{x\}$ in the definition of $c_{xy}$ gives
\begin{equation}\label{ad15}
 g(R)-g(R')=c_{xy}=b_x-b_y.
\end{equation}

Let $R,T\in\binom{[8]}{3}$.
For $R\ne T$, write
$
 R\setminus T=\{u_1,\ldots,u_s\}
$ and
$
 T\setminus R=\{v_1,\ldots,v_s\}, s\ge1.
$
Put $R_0=R$ and
$
 R_i=(R_{i-1}\setminus\{u_i\})\cup\{v_i\}, 1\le i\le s.
$
Then $R_s=T$.
At each step, $u_i\in R_{i-1}$ and $v_i\notin R_{i-1}$.
The  equality \eqref{ad15} gives
\[
 \begin{aligned}
 g(R)-g(T)
 &=\sum_{i=1}^s\bigl(g(R_{i-1})-g(R_i)\bigr)=\sum_{i=1}^s(b_{u_i}-b_{v_i})\\
 &=\sum_{z\in R}b_z-\sum_{z\in T}b_z.
 \end{aligned}
\]
This identity also holds when $R=T$.
Hence, $g(R)-\sum_{z\in R}b_z$ is constant.
Denote its value by $a$. Then
\[
 g(R)=a+\sum_{z\in R}b_z
 \qquad\text{for every }R\in\binom{[8]}{3}.
\]

Since $0<|\mathcal G|=35<56$, the function $g$ is not constant.
By Lemma~\ref{lem:boolean}, there is $x\in[8]$ such that
$
 \mathcal G\subseteq\{R\in\binom{[8]}{3}:x\in R\}
$
or
$
 \mathcal G\subseteq\{R\in\binom{[8]}{3}:x\notin R\}.
$
The first case gives $|\mathcal G|\le\binom{7}{2}=21$,
contrary to $|\mathcal G|=35$.
In the second case, the containing family has
$\binom{7}{3}=35$ members.
Hence,
$
 \mathcal G=\{R\in\binom{[8]}{3}:x\notin R\}.
$
Taking complements gives
$
 \mathcal F
 =\{A\in\binom{[8]}{5}:x\in A\}.
$
\end{proof}

\subsection{Proof of Theorem~\ref{thm:main}}

We recall the two known results used in the induction.

\begin{theorem}[Currier~\cite{Currier}]
\label{thm:Currier}
Let $d\ge3$, $k\ge d+1$, and $n\ge2k-d+2$.
If $\mathcal F\subseteq\binom{[n]}{k}$ contains no
$d$-simplex-cluster, then
\[
 |\mathcal F|\le\binom{n-1}{k-1}.
\]
Equality holds if and only if there exists $x\in[n]$ such that
$
 \mathcal F=\{A\in\binom{[n]}{k}:x\in A\}.
$
\end{theorem}

\begin{theorem}[Mubayi--Verstra\"ete~\cite{MV}]
\label{thm:MV}
Let $k\ge3$ and $n\ge3k/2$.
If $\mathcal F\subseteq\binom{[n]}{k}$ contains no
$2$-simplex, then
\[
 |\mathcal F|\le\binom{n-1}{k-1}.
\]
Equality holds if and only if there exists $x\in[n]$ such that
$
 \mathcal F=\{A\in\binom{[n]}{k}:x\in A\}.
$
\end{theorem}

We now combine the preceding results to prove
Theorem~\ref{thm:main}.

\begin{proof}[Proof of Theorem~\ref{thm:main}]
We proceed by induction on $d$.
For $d=2$, we have $n\ge3k/2$.
When $n\le2k$, every $2$-simplex is a $2$-simplex-cluster.
For $k=3$, a $2$-simplex $\{A,B,C\}$ has empty total
intersection and nonempty pairwise intersections.
Hence,
\[
 |A\cup B\cup C|
 =9-|A\cap B|-|A\cap C|-|B\cap C|\le6=2k.
\]
Thus, every $2$-simplex is a $2$-simplex-cluster
when $n\le2k$ or $k=3$.
Since $\mathcal F$ contains no $2$-simplex-cluster,
it contains no $2$-simplex in these cases.
Theorem~\ref{thm:MV} then gives the required bound
and equality characterization.
The remaining case $k\ge4$ and $n>2k$ follows
from Theorem~\ref{thm:triangle}.

Let $d\ge3$, and assume that the theorem holds for $d-1$.
For $n\ge2k-d+2$, the conclusion follows from
Theorem~\ref{thm:Currier}.

Suppose next that $n=2k-d+1$, and put $r=n-k$.
Then $k=r+d-1$.
The assumptions $n\ge k(d+1)/d$  and $k>d$ give
$r\ge k/d>1$, so $r\ge2$.
Here $n<2k$, so every $d$-simplex is a
$d$-simplex-cluster.
Since $\mathcal F$ contains no $d$-simplex-cluster,
it contains no $d$-simplex.
If $r=2$, then
$
 n=d+3\geq 6
$
and
$ k=d+1=n-2.$
Since $d=n-3$, Lemma~\ref{lem:ranktwo} applies.
For $r\ge3$ and $(d,r)\ne(3,3)$, we have
$d(r-1)>2r$, so Lemma~\ref{lem:transition} applies.
The remaining case $d=r=3$ gives
$(d,k,n)=(3,5,8)$.
Here $\mathcal F$ contains no $3$-simplex,
so Lemma~\ref{lem:eight} applies.

It remains to treat $n\le2k-d$.
Put $r=n-k$ and $s=k-r$.
Since $n\ge k(d+1)/d$ and $k>d$, we have
$
 r=n-k\ge k/d>1.
$
Thus, $r\ge2$.
Since $n\le2k-d$, we also have
$
 s=k-r=2k-n\ge d\geq 3.
$
Hence,
$
 k=r+s\geq r+3, n=k+r=2r+s\geq 2r+3.
$
Since $n<2k$, the family $\mathcal F$ contains
no $d$-simplex.
Define
\[
 \mathcal G=\{[n]\setminus A:A\in\mathcal F\}
 \subseteq\binom{[n]}{r}.
\]
For $E\in\mathcal G$, put $X=[n]\setminus E$ and
\[
 \mathcal G_E
 =\{B\in\mathcal G:B\cap E=\emptyset\}
 \subseteq\binom{X}{r}.
\]

We claim that the sets $X\setminus B$,
$B\in\mathcal G_E$, form an $s$-uniform family
containing no $(d-1)$-simplex.
Suppose that $B_1,\ldots,B_d\in\mathcal G_E$
are such that $X\setminus B_1,\ldots,X\setminus B_d$
form a $(d-1)$-simplex.
Put
\[
 A_0=X,\qquad
 A_i=[n]\setminus B_i\quad(i\in[d]).
\]
These are distinct members of $\mathcal F$.
Since $B_i\subseteq X$, we have
$A_i=E\cup(X\setminus B_i)$.
The simplex conditions give
\[
 \begin{aligned}
 \bigcap_{i=0}^d A_i
 &=\bigcap_{i=1}^d(X\setminus B_i)=\emptyset,\quad
 \bigcap_{i=1}^d A_i
 =E\ne\emptyset,\\
 \bigcap_{0\le j\le d, j\ne i}A_j
 &=\bigcap_{j\in[d]\setminus\{i\}}(X\setminus B_j)
 \ne\emptyset
 \qquad(i\in[d]).
 \end{aligned}
\]
Thus, $A_0,\ldots,A_d$ form a $d$-simplex,
contradicting the assumption on $\mathcal F$.

In particular, the family
$\{X\setminus B:B\in\mathcal G_E\}$ contains no
$(d-1)$-simplex-cluster.
Its members have size $s\ge d$, and its ground set
$X$ has size $k$.
The original parameter condition gives $(d+1)r\ge n$,
and
\[
 (d-1)k-ds=dr-k=(d+1)r-n\ge0.
\]
Hence, $k\ge ds/(d-1)$, so the induction hypothesis
applies and gives
\[
 |\mathcal G_E|
 \le\binom{k-1}{s-1}
 =\binom{k-1}{r}
 =\frac{k-r}{k}\binom{k}{r}.
\]

Put
\[
 m=|\mathcal G|=|\mathcal F|,\qquad
 N=\binom{n}{r},\qquad
 \alpha=\frac mN,\qquad
 D=\binom{k}{r}.
\]
If $m=0$, the desired bound is immediate.
Assume $m>0$.
The number of ordered disjoint pairs in $\mathcal G$
is $\sum_{E\in\mathcal G}|\mathcal G_E|$.
Applying Lemma~\ref{lem:kneser} and summing the
preceding upper bound over $E\in\mathcal G$, we obtain
\[
 mD\left(\frac nk\alpha-\frac rk\right)
 \le\sum_{E\in\mathcal G}|\mathcal G_E|
 \le mD\frac{k-r}{k}.
\]
Dividing by $mD$ gives
\[
 \frac nk\alpha-\frac rk\le\frac{k-r}{k},
\]
which reduces to $\alpha\le k/n$.
Consequently,
\[
 |\mathcal F|=m
 \le\frac kn\binom{n}{r}
 =\binom{n-1}{k-1}.
\]

If $|\mathcal F|=\binom{n-1}{k-1}$, then
$\alpha=k/n$, and equality holds throughout the
bound on the number of ordered disjoint pairs.
The equality assertion in Lemma~\ref{lem:kneser}
shows that the indicator function of $\mathcal G$
has the form stated in Lemma~\ref{lem:boolean}.
Since $0<|\mathcal G|<N$, its indicator function is not constant.
By Lemma~\ref{lem:boolean}, there exists $x\in[n]$ such that
$
 \mathcal G\subseteq\{R\in\binom{[n]}{r}:x\in R\}
$
or
$
 \mathcal G\subseteq\{R\in\binom{[n]}{r}:x\notin R\}.
$
The first case gives
\[
 |\mathcal G|
 \le\binom{n-1}{r-1}
 =\frac rnN
 <\frac knN
 =|\mathcal G|,
\]
a contradiction.
In the second case, the containing family has
$
 \binom{n-1}{r}
 =\frac knN
 =|\mathcal G|
$
members, so the inclusion is an equality.
Hence,
$
 \mathcal G
 =\{R\in\binom{[n]}{r}:x\notin R\}.
$
It follows that
$
 \mathcal F
 =\{A\in\binom{[n]}{k}:x\in A\}.
$
Conversely, a full star has $\binom{n-1}{k-1}$ members.
Every $d+1$ of its members have nonempty total intersection,
so it contains no $d$-simplex-cluster.
\end{proof}

\section*{Declaration of competing interest}
The authors declare that they have no competing interests.

\section*{Data availability}
No data was used for the research described in the article.

\section*{Acknowledgments}
The authors used ChatGPT as an auxiliary tool to discuss
mathematical details, assist in developing and refining proofs,
and improve the clarity and organization of the manuscript.
The authors wrote and revised the manuscript, carefully checked
the final proofs, and take full responsibility for its content.

\end{document}